\documentclass{amsproc}

\usepackage{verbatim} 
\usepackage{amsmath}
\usepackage{amsfonts}
\usepackage{amssymb}
\usepackage{mathrsfs}
\usepackage{amsthm}
\usepackage{newlfont}

\usepackage{fullpage,amsmath, amssymb,amscd}
\usepackage{latexsym, epsfig, color}
\usepackage[mathscr]{eucal}
\usepackage{xypic, graphicx}
\usepackage{amscd}
\usepackage[all, cmtip]{xy}

\usepackage{float, hyperref, nccmath, multirow, caption, subcaption}

\newcommand\cyr{%
 \renewcommand\rmdefault{wncyr}%
 \renewcommand\sfdefault{wncyss}%
 \renewcommand\encodingdefault{OT2}%
\normalfont\selectfont} \DeclareTextFontCommand{\textcyr}{\cyr}

\newtheorem{theorem}{Theorem}
\newtheorem{lemma}[theorem]{Lemma}
\newtheorem{corollary}[theorem]{Corollary}

\def\Z{\mathbb Z}

\def\Q{\mathbb Q}

\def\F{\mathbb F}

\def\Ker{\operatorname{Ker}}

\def\End{\operatorname{End}}

\def\Gal{\operatorname{Gal}}

\def\mod{\operatorname{mod}}

\def\car{\operatorname{char}}

\def\gcd{\operatorname{gcd}}

\def\GL{\operatorname{GL}}

\def\PGL{\operatorname{PGL}}

\def\sep{\operatorname{sep}}

\def\ds{\displaystyle}

\def\Quot{\operatorname{Quot}}

\def\Scal{\operatorname{Scal}}

\begin{document}

\title{
Degree bounds for projective division fields 
associated to elliptic modules with a trivial endomorphism ring
}

\date{February 19,  2020}

\author{
Alina Carmen Cojocaru and Nathan Jones
}
\address[Alina Carmen  Cojocaru]{
\begin{itemize}
\item[-]
Department of Mathematics, Statistics and Computer Science, University of Illinois at Chicago, 851 S Morgan St, 322
SEO, Chicago, 60607, IL, USA;
\item[-]
Institute of Mathematics  ``Simion Stoilow'' of the Romanian Academy, 21 Calea Grivitei St, Bucharest, 010702,
Sector 1, Romania
\end{itemize}
} \email[Alina Carmen  Cojocaru]{cojocaru@uic.edu}

\address[Nathan Jones]{
\begin{itemize}
\item[-]
Department of Mathematics, Statistics and Computer Science, University of Illinois at Chicago, 851 S Morgan St, 322
SEO, Chicago, 60607, IL, USA;
\end{itemize}
} \email[Nathan Jones]{ncjones@uic.edu}

\renewcommand{\thefootnote}{\fnsymbol{footnote}}
\renewcommand{\thefootnote}{\arabic{footnote}}

\renewcommand{\thefootnote}{\fnsymbol{footnote}}
\renewcommand{\thefootnote}{\arabic{footnote}}

\thanks{A.C.C. was partially supported  by  a Collaboration Grant for Mathematicians from the Simons Foundation  
under Award No. 318454.}

\begin{abstract}
Let $k$ be a global field,
let $A$ be a Dedekind domain with $\Quot(A) = k$,
and
let  $K$ be a finitely generated field.
Using a unified approach for both elliptic curves and Drinfeld modules $M$ defined over $K$ and
having a trivial endomorphism ring, 
with $k= \Q$, $A = \Z$ in the former  case 
and $k$ a global function field, $A$ its ring of functions regular away from a fixed prime  in the latter case,
for any nonzero ideal $\mathfrak{a} \lhd A$
we prove best possible estimates 
in the norm $|\mathfrak{a}|$ for the degrees over $K$ of the  subfields of the $\mathfrak{a}$-division fields of $M$
fixed by scalars.
\end{abstract}

\maketitle


\section{Introduction}

In the theory of elliptic modules -- elliptic curves and Drinfeld modules -- division fields play a fundamental
role; their algebraic properties (e.g., ramification, degree, and Galois group structure) are intimately related
to properties of Galois representations and are  essential to global and local questions about  elliptic modules themselves. Among the subfields of the division fields of an elliptic module, those fixed by the scalars are of special
significance. For example, as highlighted in \cite[Chapter 5]{Ad01} and \cite[Section 3]{CoDu04},
 in the case of an elliptic curve $E$ defined over $\Q$ and a positive integer $a$, 
 the subfield $J_a$ of the $a$-division field $\Q(E[a])$ fixed by the scalars of $\Gal(\Q(E[a]/\Q)) \leq \GL_2(\Z/a \Z)$ is closely related to the modular curve 
$X_0(a)$ which parametrizes cyclic isogenies of degree $a$ between elliptic curves;
indeed, $J_a$ may be interpreted as the splitting field of the modular polynomial $\Phi_a(X, j(E))$ (see \cite[Section 69]{We08} and \cite[Section 11.C]{Co89} for the properties of the modular polynomials $\Phi_a(X, Y)$). 
The arithmetic properties of the family of fields $(J_a)_{a \geq 1}$  are closely related 
to properties of  the reductions  $E(\mod p)$ of $E$ modulo primes $p$, 
including
the growth of the order of the Tate-Shafarevich group of  the curve $E(\mod p)$ when  viewed as constant over its own function field 
(see \cite{CoDu04})
and
the growth of the absolute discriminant of the endomorphism ring of the curve $E(\mod p)$ when viewed over the finite field $\F_p$ 
(see \cite{CoFi19}).
An essential ingredient when deriving properties about $E(\mod p)$ from the fields $J_a$ is the growth of the degrees
$[J_a :\Q]$. The goal of this article is to prove best possible estimates in $a$ for the degrees of such fields in the unified setting of elliptic curves and Drinfeld modules with a trivial endomorphism ring.

To state our main result, we proceed as in \cite{Br10} and fix:
$k$ a global field,
$A$ a Dedekind domain with $\Quot(A) = k$,
$K$ a finitely generated field, 
and
$M$ a $(G_K, A)$-module of rank $r \geq 2$,
where $G_K = \Gal(K^{\sep}/K)$ denotes the absolute Galois group of $K$.
Specifically,  
$M$ is an $A$-module 
with a continuous $G_K$-action that commutes with the $A$-action
and 
with the property that, for any ideal $0 \neq \mathfrak{a} \lhd A$, the $\mathfrak{a}$-division submodule 
$$
M[\mathfrak{a}] := \left\{x \in M: \alpha x = 0  \ \forall \alpha \in \mathfrak{a}\right\}
$$ 
has $A$-module structure
$$
M[\mathfrak{a}] \simeq_A (A/\mathfrak{a})^r.
$$
The $(G_K, A)$-module structure on $M$  gives rise to a compatible system of Galois representations 
 $\rho_{\mathfrak{a}}: G_K  \longrightarrow \GL_r(A/\mathfrak{a})$
and to a continuous representation
$$
\rho : G_K \longrightarrow \GL_r(\hat{A}),
$$
where $\hat{A} :=  \ds\lim_{ \leftarrow \atop{\mathfrak{a} \lhd A}} A/\mathfrak{a}$. 
Associated to these representations we have the $\mathfrak{a}$-division fields
$K_{\mathfrak{a}} :=  (K^{\sep})^{\Ker \rho_{\mathfrak{a}}}$,
for which we distinguish the subfields
$J_{\mathfrak{a}}$
fixed by the scalars
$\{\lambda I_r : \lambda \in (A/\mathfrak{a})^\times \}  \cap \Gal(K_{\mathfrak{a}}/K)\}$ 
(with $\Gal(K_{\mathfrak{a}}/K)$ viewed as a subgroup of $\GL_r(A/\mathfrak{a})$).

Denoting by
$\hat{\rho}_{\mathfrak{a}}: G_K \longrightarrow \PGL_r(A/\mathfrak{a})$ 
 the composition of the representation
 $\rho_{\mathfrak{a}}$
with the canonical projection $\GL_r(A/\mathfrak{a}) \longrightarrow \PGL_r(A/\mathfrak{a})$,
 we observe that
 $J_{\mathfrak{a}} = (K^{\sep})^{\Ker \hat{\rho}_{\mathfrak{a}}}$
 and we deduce that $[J_{\mathfrak{a}} :K]  \leq  \left|\PGL_r(A/\mathfrak{a})\right|$. Our main result provides a lower bound
 for $[J_{\mathfrak{a}} :K]$ of the same order of growth as  $\left|\PGL_r(A/\mathfrak{a})\right|$, as follows:

\begin{theorem}\label{Thm1}
We keep the above setting and assume that
\begin{equation}\label{assumption}
\left|\GL_r(\hat{A}) : \rho(G_K)\right| < \infty.
\end{equation}
Then, for any ideal $0 \neq \mathfrak{a} \lhd A$,
\begin{equation}\label{final-bound}
|\mathfrak{a}|^{r^2 - 1}
\ll_{M, K}
[J_{\mathfrak{a}} :K] 
\leq 
|\mathfrak{a}|^{r^2 - 1},
\end{equation}
where $|\mathfrak{a}| := |A/\mathfrak{a}|$.
\end{theorem}

By specializing the above general setting to  elliptic curves and to Drinfeld modules, we obtain:
\begin{corollary}\label{Cor2}
Let $K$ be a finitely generated field with $\car K = 0$
and let $E/K$ be an elliptic curve over $K$ with
$\End_{\overline{K}}(E) \simeq \Z$.
Then, for any integer $a \geq 1$, the degree $[J_a : K]$ of
the subfield $J_a$ of the $a$-division field $K_a := K(E[a])$ fixed by the scalars of $\Gal(K(E[a])/K)$ satisfies
\begin{equation}\label{final-bound-ec}
a^3
\ll_{E, K}
[J_a :K] 
\leq 
a^3.
\end{equation}
\end{corollary}

\begin{corollary}\label{Cor3}
Let $k$ be a global function field, 
let $\infty$ be a fixed place of $k$,
let $A$ be the ring of elements of $k$ regular away from $\infty$,
let $K$ be a finitely generated $A$-field with 
$A$-$\car K = 0$ 
(i.e. $k \subseteq K$),
and 
let $\psi: A \longrightarrow K\{\tau\}$ be a Drinfeld $A$-module over $K$ of rank $r \geq 2$
with $\End_{\overline{K}}(\psi) \simeq A$.
Then, for any ideal $0 \neq \mathfrak{a} \lhd A$, 
the degree $[J_{\mathfrak{a}} : K]$ of the subfield $J_{\mathfrak{a}}$ of the $\mathfrak{a}$-division field 
$K_{\mathfrak{a}} := K(\psi[\mathfrak{a}])$ fixed by the scalars of 
$\Gal(K(\psi[\mathfrak{a}])/K)$
satisfies
\begin{equation}\label{final-bound-dm}
|\mathfrak{a}|^{r^2 - 1}
\ll_{\psi, K}
[J_{\mathfrak{a}} :K] 
\leq 
|\mathfrak{a}|^{r^2 - 1}.
\end{equation}
\end{corollary}

The proof of Theorem \ref{Thm1} relies 
on consequences of assumption (\ref{assumption}), 
on applications of Goursat's Lemma, 
as well as 
on vertical growth estimates for open subgroups of $\GL_r$. 
Specializing to elliptic curves and to Drinfeld modules,
assumption  (\ref{assumption}) is essentially Serre's Open Image Theorem \cite{Se72} 
and, respectively, Pink-R\"{u}tsche's Open Image Theorem \cite{PiRu09}. 
Variations of these open image theorems also hold for elliptic curves and Drinfeld modules with nontrivial endomorphism rings.
While these complementary cases are treated unitarily in \cite{Br10} when investigating the growth of  torsion, 
when investigating the growth of $[J_{\mathfrak{a}}:K]$ 
they face particularities whose treatment we relegate to future work.

We emphasize that the upper bound in Theorem \ref{Thm1} always holds and does not necessitate assumption  (\ref{assumption}). In contrast, the lower bound  in Theorem \ref{Thm1} is intimately related to assumption (\ref{assumption}).
Indeed, one consequence of (\ref{assumption}) is that there exists an ideal 
${\mathfrak{a}}(M, K) \unlhd A$,
which (a priori) depends on $M$ and $K$ and which has the property that, for any prime ideal $\mathfrak{l} \nmid \mathfrak{a}(M, K)$,
$\Gal(J_{\mathfrak{l}}/K) \simeq \PGL_r(A/\mathfrak{l})$. Then, for such an ideal $\mathfrak{l}$,
the lower  bound in (\ref{final-bound}) follows immediately.
The purpose of Theorem \ref{Thm1} is to prove similar lower bounds  for {\it{all}} ideals $\mathfrak{a} \lhd A$.

The dependence of the lower bound  in (\ref{final-bound}) on $M$ (which also includes dependence on $r$)
 and on $K$ is an important topic related to the uniform boundedness of the torsion of $M$;
 while we do not address it in the present paper, we refer the reader to \cite{Br10} and \cite{Jo19} 
 for related discussions and for additional references.

The fields $J_\mathfrak{a}$ play a prominent role in a multitude of problems, such as
in deriving 
non-trivial upper bounds for the number of non-isomorphic Frobenius fields 
associated to an elliptic curve and, respectively, to a  Drinfeld module
(see \cite{CoDa08ec} and \cite{CoDa08dm});
in investigating the  discriminants of the endomorphism rings  of the reduction of an elliptic curve and, respectively,  of a Drinfeld module
(see \cite{CoFi19} and \cite{CoPa20});
and 
in proving non-abelian reciprocity laws for primes and, respectively, for irreducible polynomials 
(see \cite{DuTo02}, \cite{CoPa15}, and \cite{GaPa19}).
For such applications, an essential piece of information is the growth of the degree $[J_{\mathfrak{a}} : K]$
as a function of the norm $|\mathfrak{a}|$.
For example, Corollary \ref{Cor2} is a key ingredient in proving that, 
for any elliptic curve $E/\Q$ with $\End_{\overline{\Q}}(E) \simeq \Z$, provided the Generalized Riemann Hypothesis holds for the division fields  of $E$,  
there exists a set of primes $p$ of natural density 1 with  the property that the absolute discriminant of the imaginary quadratic order $\End_{\F_p}(E)$ is as close as possible to its natural upper bound; 
see \cite[Theorem 1]{CoFi19}.
Similarly, Corollary \ref{Cor3} is a key ingredient in proving that, 
denoting by $\F_q$ the finite field with $q$ elements and assuming that $q$ is odd,
for any generic Drinfeld module $\psi: \F_q[T] \longrightarrow \F_q(T)\{\tau\}$ of rank 2 and with 
$\End_{\overline{\F_q(T)}}(\psi) \simeq \F_q[T]$, 
there exists a set of prime ideals $\mathfrak{p} \lhd \F_q[T]$ of Dirichlet density 1 with  the property that the norm of  the discriminant of the imaginary quadratic order $\End_{\F_{\mathfrak{p}}}(\psi)$ is as close as possible to its natural upper bound; 
see \cite[Theorem 6]{CoPa20}. We expect that Theorem \ref{Thm1} will be of use to other arithmetic studies of elliptic modules.

{\bf{Notation}}. 
In what follows, we use the standard $\ll$, $\gg$, and $\asymp$ notation:
given suitably defined real functions $h_1, h_2$,
we say that
$h_1 \ll h_2$ or $h_2 \gg h_1$
if 
$h_2$ is positive valued 
and
 there exists a positive constant $C$ such that 
$|h_1(x)| \leq C h_2(x)$ for all $x$ in the domain of $h_1$;
we say that
$h_1 \asymp h_2$ 
if 
$h_1$, $h_2$ are positive valued 
and 
$h_1 \ll h_2 \ll h_1$;
we say that
$h_1 \ll_D h_2$ or $h_2 \gg_D h_1$
if 
$h_1 \ll h_2$
and
the implied $\ll$-constant $C$ depends on priorly given data  $D$;
similarly,
we say that 
$h_1 \asymp_D h_2$ 
if the implied constant in either one of the $\ll$-bounds
$h_1 \ll h_2 \ll h_1$
depends on priorly  given data $D$.
We also use the standard divisibility notation for ideals in a Dedekind domain. In particular, 
given two ideals $\mathfrak{a}$, $\mathfrak{b}$, 
we write $\mathfrak{a} \mid \mathfrak{b}^{\infty}$ if 
all the prime ideal factors of $\mathfrak{a}$ are among the prime ideal factors of $\mathfrak{b}$
(with possibly different exponents).
Further notation will be introduced over the course of the paper as needed.

\section{Goursat's Lemma and variations}

In this section we recall Goursat's Lemma on fibered products of groups (whose definition we also recall shortly) 
and detail the behavior of such fibered products under intersection.

\begin{lemma} \label{goursat} (Goursat's Lemma)

Let $G_1$, $G_2$ be groups 
and 
for $i \in \{1, 2 \}$ denote by $\pi_i : G_1 \times G_2 \longrightarrow G_i$ the projection map onto the $i$-th factor. 
Let $G \leq G_1 \times G_2$ be a subgroup and 
assume that  $\pi_1(G) = G_1$, $\pi_2(G) = G_2$.
Then there exist a group $\Gamma$ and a pair of surjective group homomorphisms 
$\psi_1 : G_1 \longrightarrow \Gamma$, 
$\psi_2 : G_2 \longrightarrow \Gamma$
such that 
$$G = G_1 \times_\psi G_2 := \{ (g_1,g_2) \in G_1 \times G_2 : \psi_1(g_1) = \psi_2(g_2) \}.$$
\end{lemma}
\begin{proof}
See \cite[Lemma 5.2.1]{Ri76}.
\end{proof}

We call $G_1 \times_\psi G_2$ the \emph{fibered product of $G_1$ and $G_2$ over $\psi := (\psi_1, \psi_2)$}.  
The next lemma details what happens when we intersect such a fibered product with a subgroup of the form $H_1 \times H_2$ defined by subgroups $H_1 \leq G_1$ and $H_2 \leq G_2$. 

 It is clear that
$$
\left( H_1 \times H_2 \right) \cap \left( G_1 \times_\psi G_2 \right) = H_1 \times_{\psi} H_2 := \{ (h_1, h_2) \in H_1 \times H_2 : \psi_1(h_1) = \psi_2(h_2) \}.
$$
However, this representation does not specify the restricted common quotient inside $\Gamma$.
In particular, it can be the case that the fibered product $H := H_1 \times_{\psi} H_2$ 
does \emph{not} satisfy $\pi_i(H) = H_i$ for each $i \in \{1, 2\}$.  The following lemma clarifies this situation.
\begin{lemma} \label{goursat-variation}
Let $G_1$, $G_2$ be groups,
 let $\psi_1 : G_1 \rightarrow \Gamma$, $\psi_2 : G_2 \rightarrow \Gamma$ be surjective group homomorphisms onto a group $\Gamma$, 
and 
let $G_1 \times_\psi G_2$ be the associated fibered product.  
Furthermore, let $H_1 \leq G_1$, $H_2 \leq G_2$ be subgroups.  
Define the subgroup
$$\Gamma_H := \psi_1(H_1) \cap \psi_2(H_2) \leq \Gamma.$$
Then
\begin{equation} \label{fiberedintersectionequality}
\left( H_1 \times H_2 \right) \cap \left( G_1 \times_\psi G_2 \right) 
=
\left( H_1 \cap \psi_1^{-1}(\Gamma_H) \right) \times_{\psi} \left( H_2 \cap \psi_2^{-1}(\Gamma_H) \right)
\end{equation}
and the canonical projection maps
\[
\begin{split}
\pi_1 : &\left( H_1 \cap \psi_1^{-1}(\Gamma_H) \right) \times_\psi \left( H_2 \cap \psi_2^{-1}(\Gamma_H) \right) \longrightarrow  H_1 \cap \psi_1^{-1}(\Gamma_H), \\
\pi_2 : &\left( H_1 \cap \psi_1^{-1}(\Gamma_H) \right) \times_\psi \left( H_2 \cap \psi_2^{-1}(\Gamma_H) \right) \longrightarrow  H_2 \cap \psi_2^{-1}(\Gamma_H)
\end{split}
\]
are surjective.
\end{lemma}
\begin{proof}
We first establish \eqref{fiberedintersectionequality}.  
Since the containment ``$\supseteq$'' is immediate, we only need to establish ``$\subseteq$.''  Let $(h_1,h_2) \in \left( H_1 \times H_2 \right) \cap \left( G_1 \times_\psi G_2 \right)$, i.e. $h_1 \in H_1$, $h_2 \in H_2$, and $\psi_1(h_1) = \psi_2(h_2)$.  From the definition of $\Gamma_H$, it follows that $\psi_1(h_1) = \psi_2(h_2) \in \Gamma_H$.
Thus $(h_1,h_2) \in \left( H_1 \cap \psi_1^{-1}(\Gamma_H) \right) \times_{\psi} \left( H_2 \cap \psi_2^{-1}(\Gamma_H) \right)$, establishing \eqref{fiberedintersectionequality}.

To see why the projection map 
\begin{equation} \label{surjectivityofpi1}
\pi_1 : H_1 \cap \psi_i^{-1}(\Gamma_H) \longrightarrow H_1 \cap \psi_1^{-1}(\Gamma_H)
\end{equation}
is surjective, fix $h_1 \in H_1 \cap \psi_1^{-1}(\Gamma_H)$ and set $\gamma := \psi_1(h_1) \in \Gamma_H$.  By the definition of $\Gamma_H$, we find $h_2 \in H_2$ with $\psi_2(h_2) = \gamma$.
Thus  $(h_1,h_2) \in \left( H_1 \cap \psi_1^{-1}(\Gamma_H) \right) \times_\psi \left( H_2 \cap \psi_2^{-1}(\Gamma_H) \right)$ and $\pi_1(h_1,h_2) = h_1$, proving the surjectivity of $\pi_1$ in \eqref{surjectivityofpi1}.  
The surjectivity of $\pi_2$ is proved similarly.
\end{proof}

\section{Proof of Theorem \ref{Thm1}}

In this section we prove Theorem \ref{Thm1}.
We will make use of the following notation:
$$G := \rho(G_K) \leq \GL_r(\hat{A});$$
for any ideal $0 \neq \mathfrak{a} \lhd A$,
we write
$$G(\mathfrak{a}) := \rho_{\mathfrak{a}}(G_K)
 \leq
 \GL_r(A/\mathfrak{a});$$
for any subgroup $H \leq \GL_r(A/\mathfrak{a})$, we write
$$\Scal_H := H \cap \{\alpha I_r: \alpha \in (A/\mathfrak{a})^{\times} \}.$$
With this notation,  we see that
$J_{\mathfrak{a}} = K(E[\mathfrak{a}])^{\Scal_{G(\mathfrak{a})}}$.

To prove the theorem,
let $0 \neq \mathfrak{a} \lhd A$ be a fixed arbitrary ideal. 
The proof of the upper bound  is an immediate consequence to the
injection
$\Gal(J_{\mathfrak{a}}/K) \hookrightarrow \PGL_r(A/\mathfrak{a})$
defined  by $\hat{\rho}_{\mathfrak{a}}$.
Indeed,
using that
$$
\left|\PGL_r(A/\mathfrak{a})\right| 
= 
\frac{1}{\left|(A/\mathfrak{a})^{\times}\right|} \left|\GL_r(A/\mathfrak{a})\right|,
$$
$$
|(A/\mathfrak{a})^{\times}| 
= 
|\mathfrak{a}| \ds\prod_{\mathfrak{p} \mid \mathfrak{a}} \left(1 - \frac{1}{|\mathfrak{p}|}\right),
$$
and
$$
\left|\GL_r(A/\mathfrak{a})\right|
=
|\mathfrak{a}|^{r^2 }
\ds\prod_{
\mathfrak{p} \mid \mathfrak{a}
\atop{\mathfrak{p} \ \text{prime}}
}
\left(1 - \frac{1}{\left|\mathfrak{p}\right|}\right)
\left(1 - \frac{1}{\left|\mathfrak{p}\right|^2}\right)
\ldots
\left(1 - \frac{1}{\left|\mathfrak{p}\right|^r}\right)
$$
(see \cite[Lemma 2.3, p. 1244]{Br10} for the latter),
we obtain that
\begin{eqnarray*}
[J_{\mathfrak{a}} : K]
\leq
|\PGL_r(A/\mathfrak{a})|
=
|\mathfrak{a}|^{r^2 - 1}
\ds\prod_{
\mathfrak{p} \mid \mathfrak{a}
\atop{\mathfrak{p} \ \text{prime}}
}
\left(1 - \frac{1}{\left|\mathfrak{p}\right|^2}\right)
\ldots
\left(1 - \frac{1}{\left|\mathfrak{p}\right|^r}\right)
\leq
|\mathfrak{a}|^{r^2 - 1}.
\end{eqnarray*}
The proof of the lower bound  relies on  several consequences to assumption (\ref{assumption}),
as well as on applications of Goursat's Lemma \ref{goursat} and its variation Lemma \ref{goursat-variation},  as detailed below.

Thanks to (\ref{assumption}), there exists an ideal $\mathfrak{m} = \mathfrak{m}_{M, K} \unlhd A$ such that
\begin{equation}\label{def-m}
G = \pi^{-1}(G(\mathfrak{m})),
\end{equation}
where
$\pi: \GL_r(\hat{A}) \longrightarrow \GL_r(A/\mathfrak{m})$ is the canonical projection.
We take $\mathfrak{m}$ to be the smallest such ideal with respect to divisibility
and we write its unique prime ideal factorization as
$\mathfrak{m} 
= 
\ds\prod_{\mathfrak{p}^{v_{\mathfrak{p}}(\mathfrak{m}) } || \mathfrak{m}} 
\mathfrak{p}^{v_{\mathfrak{p}}(\mathfrak{m}) },$
where each exponent satisfies $v_{\mathfrak{p}}(\mathfrak{m}) \geq 1$.

With the ideal $\mathfrak{m}$ in mind, 
we write the arbitrary ideal $\mathfrak{a}$ uniquely as
\begin{equation}\label{factor-a}
\mathfrak{a} = \mathfrak{a}_1 \mathfrak{a}_2,
\end{equation}
where
\begin{equation}\label{def-a1}
\mathfrak{a}_1 \mid \mathfrak{m}^{\infty},
\end{equation}
\begin{equation}\label{def-a2}
\gcd(\mathfrak{a}_2, \mathfrak{m}) = 1.
\end{equation}
For future use, we record that
\begin{equation}\label{gcd-a1-a2}
\gcd(\mathfrak{a}_1, \mathfrak{a}_2) =  1.
\end{equation}
We  also write the ideal $\mathfrak{a}_1$ uniquely as
\begin{equation*}
\mathfrak{a}_1 = \mathfrak{a}_{1, 1} \ \mathfrak{a}_{1, 2},
\end{equation*}
where
$\mathfrak{a}_{1, 1} = 
\ds\prod_{
\mathfrak{p}^{e_{\mathfrak{p}}} || {\mathfrak{a}}_{1, 1}
\atop{
e_{\mathfrak{p}} > v_{\mathfrak{p}}(\mathfrak{m})
}
} 
\mathfrak{p}^{e_{\mathfrak{p}}}$
and
$\mathfrak{a}_{1, 2} = 
\ds\prod_{
\mathfrak{p}^{f_{\mathfrak{p}}} || {\mathfrak{a}}_{1, 1}
\atop{
f_{\mathfrak{p}} \leq v_{\mathfrak{p}}(\mathfrak{m})
}
} 
\mathfrak{p}^{f_{\mathfrak{p}}}$.
Note that
\begin{equation}\label{gcd-a11-a12}
\gcd(\mathfrak{a}_{1, 1}, \mathfrak{a}_{1, 2}) = 1,
\end{equation}
\begin{equation}\label{def-a11}
\mathfrak{a}_{1, 1} \mid \mathfrak{m}^\infty
\end{equation}
and
\begin{equation}\label{def-a12}
\mathfrak{a}_{1, 2} \mid \mathfrak{m}.
\end{equation}

Under the isomorphism of the Chinese Remainder Theorem, 
we deduce from  (\ref{def-m}) that
\begin{equation}\label{G-G1-G2}
G(\mathfrak{a}) \simeq G(\mathfrak{a}_1) \times \GL_r(A/\mathfrak{a}_2)
\end{equation}
and, consequently,  that there exist group isomorphisms
\begin{equation}\label{Scal1-Scal2}
\Scal_{G(\mathfrak{a})} 
\simeq 
{\Scal_{G(\mathfrak{a}_1)} \times {\Scal_{\GL_r(A/\mathfrak{a}_2)}}},
\end{equation}
\begin{equation}\label{Ga/Scala}
G(\mathfrak{a})/\Scal_{G(\mathfrak{a})} 
\simeq 
\left(G(\mathfrak{a}_1)/\Scal_{G(\mathfrak{a}_1)}\right) \times \PGL_r(A/\mathfrak{a}_2).
\end{equation}

Next, applying Lemma \ref{goursat} to the groups 
$G = G(\mathfrak{a}_1)$,
$G_1 = G(\mathfrak{a}_{1, 1})$, and $G_2 = G(\mathfrak{a}_{1, 2})$, 
we deduce that
there exist 
a group $\Gamma$ 
and 
surjective group homomorphisms
$\psi_1: G(\mathfrak{a}_{1, 1}) \longrightarrow \Gamma$, 
$\psi_2: G(\mathfrak{a}_{1, 2}) \longrightarrow \Gamma$,
which give rise to a group isomorphism
\begin{equation}\label{G-a1}
G(\mathfrak{a}_1) \simeq G(\mathfrak{a}_{1, 1}) \times_{\psi} G(\mathfrak{a}_{1, 2}).
\end{equation}
Furthermore, applying Lemma \ref{goursat-variation} to the subgroups
$H_1 = \Scal_{G(\mathfrak{a}_{1, 1})}$ and $H_2 = \Scal_{G(\mathfrak{a}_{1, 2})}$,
we deduce that
there exists a group isomorphism
\begin{equation}\label{scal-products}
\left(
\Scal_{G(\mathfrak{a}_{1, 1})} \times \Scal_{G(\mathfrak{a}_{1, 2})}
\right)
\cap
\left(
G(\mathfrak{a}_{1, 1}) \times_{\psi} G(\mathfrak{a}_{1, 2})
\right)
\simeq
\left(
\Scal_{G(\mathfrak{a}_{1, 1})} \cap \psi_1^{-1} (\Gamma_{\Scal})
\right)
\times_{\psi}
\left(
\Scal_{G(\mathfrak{a}_{1, 2})} \cap \psi_2^{-1} (\Gamma_{\Scal})
\right),
\end{equation}
where
$$
\Gamma_{\Scal} 
:= 
\psi_1\left(\Scal_{G(\mathfrak{a}_{1, 1})}\right) \cap \psi_1\left(\Scal_{G(\mathfrak{a}_{1, 2})}\right)
\leq 
\Gamma.
$$

From (\ref{Ga/Scala}) we derive that
\begin{equation}\label{Ja-degree-first}
[J_{\mathfrak{a}} :K] 
= 
\left|G(\mathfrak{a})/\Scal_{G(\mathfrak{a})}\right| 
= 
\left|G(\mathfrak{a}_1)/\Scal_{G(\mathfrak{a}_1)}\right| \cdot \left|\PGL_r(A/\mathfrak{a}_2)\right|.
\end{equation}
Then, using  
(\ref{G-a1}) and (\ref{scal-products}),
we derive that
\begin{eqnarray}\label{Ga1/Scala1}
\left|
G(\mathfrak{a}_1)/\Scal_{G(\mathfrak{a}_1)}
\right|
&=&
\frac{
\left|
G(\mathfrak{a}_{1, 1}) \times_{\psi} G(\mathfrak{a}_{1, 2})
\right|
}
{
\left|
\left(
\Scal_{G(\mathfrak{a}_{1, 1})} \times \Scal_{G(\mathfrak{a}_{1, 2})}
\right)
\cap
G(\mathfrak{a}_1)
\right|
}
\nonumber
\\
&=&
\frac{
\left|
G(\mathfrak{a}_{1, 1}) \times_{\psi} G(\mathfrak{a}_{1, 2})
\right|
}{
\left|
\left(
\Scal_{G(\mathfrak{a}_{1, 1})} \cap \psi_1^{-1}(\Gamma_{\Scal})
\right)
\times_{\psi}
\left(
\Scal_{G(\mathfrak{a}_{1, 2})} \cap \psi_2^{-1}(\Gamma_{\Scal})
\right)
\right|
}
\nonumber
\\
&=&
\frac{
\left|G(\mathfrak{a}_{1, 1})\right|
}{
\left|
\Scal_{G(\mathfrak{a}_{1, 1})} \cap \psi_1^{-1}(\Gamma_{\Scal})
\right|
}
\cdot
\frac{
\left|\Gamma_{\Scal}\right|
}{
|\Gamma|
}
\cdot
\frac{
\left|
G(\mathfrak{a}_{1, 2})
\right|
}{
\left|
\Scal_{G(\mathfrak{a}_{1, 2})} \cap \psi_2^{-1}(\Gamma_{\Scal})
\right|
}
\nonumber
\\
&=&
\frac{
\left|G(\mathfrak{a}_{1, 1})\right|
}{
\left|
\Scal_{G(\mathfrak{a}_{1, 1})} \cap \psi_1^{-1}(\Gamma_{\Scal})
\right|
}
\cdot
\frac{
\left|
\psi_1\left(\Scal_{G(\mathfrak{a}_{1, 1})}\right)
\cap
\psi_2\left(\Scal_{G(\mathfrak{a}_{1, 2})}\right)
\right|
}{
|\Gamma|
}
\cdot
\frac{
\left|
G(\mathfrak{a}_{1, 2})
\right|
}{
\left|
\Scal_{G(\mathfrak{a}_{1, 2})} \cap \psi_2^{-1}(\Gamma_{\Scal})
\right|
}.
\end{eqnarray}

Recalling (\ref{def-a12}), we deduce that the last two factors above are bounded, from above and below,
by constants depending on $\mathfrak{m}$, hence on $M$ and $K$:
\begin{equation}\label{bound-a12-part}
\frac{
\left|
\psi_1\left(\Scal_{G(\mathfrak{a}_{1, 1})}\right)
\cap
\psi_2\left(\Scal_{G(\mathfrak{a}_{1, 2})}\right)
\right|
}{
|\Gamma|
}
\cdot
\frac{
\left|
G(\mathfrak{a}_{1, 2})
\right|
}{
\left|
\Scal_{G(\mathfrak{a}_{1, 2})} \cap \psi_2^{-1}(\Gamma_{\Scal})
\right|
}
\asymp_{M, K}
1.
\end{equation}

It remains to analyze the first factor in (\ref{Ga1/Scala1}).
For this, consider the canonical projection
$$
\pi_{1, 1}: \GL_r(A/\mathfrak{a}_{1, 1}) \longrightarrow \GL_r(A/\gcd(\mathfrak{a}_{1, 1}, \mathfrak{m}))
$$
and, upon recalling \eqref{def-m},  observe that
\begin{equation}\label{Ga11}
G(\mathfrak{a}_{1, 1}) = \pi_{1, 1}^{-1}(G(\gcd(\mathfrak{a}_{1, 1}, \mathfrak{m})))
\end{equation}
and
\begin{equation}\label{Ker-pi11}
\Ker \pi_{1, 1} \subseteq \Ker \psi_1.
\end{equation}
Thus the subgroups 
$G(\mathfrak{a}_{1, 1}) \leq \GL_r(A/\mathfrak{a}_{1, 1})$ 
and
$G(\gcd(\mathfrak{a}_{1, 1}, \mathfrak{m})) \leq \GL_r(A/\gcd(\mathfrak{a}_{1,1}, \mathfrak{m}))$, 
together with the group $\Gamma$,
fit into a commutative diagram
$$
\xymatrix{
G(\mathfrak{a}_{1, 1})  
\; \; \; 
\ar@{->>}[rd]_{\psi_1}   
\ar@{->>}[r]^{\hspace*{-0.7cm} \pi_{1,1}}   
& \; \; G(\gcd(\mathfrak{a}_{1, 1}, \mathfrak{m})) 
\ar@{-->>}[d]^{\rho}
\\
                                              & \Gamma          
}
$$
in which 
the vertical map $\rho$ is some surjective group homomorphism 
and 
the horizontal map $\left.\pi_{1,1}\right|_{G(\mathfrak{a}_{1, 1})}$ is
$\left(\frac{|\mathfrak{a}_{1, 1}|}{|\gcd(\mathfrak{a}_{1,1}, \mathfrak{m})|}\right)^{r^2}$ to $1$.
Furthermore, 
the subgroups
$\psi_1^{-1}(\Gamma_{\Scal}) \cap \Scal_{G(\mathfrak{a}_{1,1})} \leq \Scal_{\GL_r(A/\mathfrak{a}_{1, 1})} \simeq (A/\mathfrak{a}_{1, 1})^{\times}$
and
$\rho^{-1}(\Gamma_{\Scal}) \cap \Scal_{G\left(\gcd(\mathfrak{a}_{1,1},\mathfrak{m})\right)} \leq \Scal_{\GL_r(A/\gcd(\mathfrak{a}_{1, 1}, \mathfrak{m}))} \simeq (A/\gcd(\mathfrak{a}_{1, 1}, \mathfrak{m}))^{\times}$,
together with the group $\Gamma_{\Scal}$,
fit into the commutative diagram
$$
\xymatrix{
\psi_1^{-1}(\Gamma_{\Scal}) \cap \Scal_{G(\mathfrak{a}_{1,1})}
\; \; \; 
\ar@{->>}[rd]_{\psi_1}   
\ar@{->>}[r]^{\hspace*{-0.1cm} \pi_{1,1}}   
& \; \; \rho^{-1}(\Gamma_{\Scal}) \cap \Scal_{G\left(\gcd(\mathfrak{a}_{1,1},\mathfrak{m})\right)}
\ar@{->>}[d]^{\rho}
\\
                                              & \Gamma_{\Scal}        
}
$$
in which 
the horizontal map $\left.\pi_{1,1}\right|_{\psi_1^{-1}(\Gamma_{\Scal})\cap \Scal_{G(\mathfrak{a}_{1,1})}}$ is
$\frac{|\mathfrak{a}_{1, 1}|}{|\gcd(\mathfrak{a}_{1,1}, \mathfrak{m})|}$ to $1$.
We deduce that
\begin{equation}\label{Ga11-order}
\left|G(\mathfrak{a}_{1, 1})\right|
=
\left(\frac{|\mathfrak{a}_{1, 1}|}{|\gcd(\mathfrak{a}_{1,1}, \mathfrak{m})|}\right)^{r^2}
\left|
G(\gcd(\mathfrak{a}_{1, 1}, \mathfrak{m}))
\right|  
\asymp_{M, K}
|\mathfrak{a}_{1, 1}|^{r^2}
\end{equation}
and
\begin{equation}\label{Scala11-order}
\left| \psi_1^{-1}(\Gamma_{\Scal}) \cap \Scal_{G(\mathfrak{a}_{1, 1})} \right|
=
\frac{|\mathfrak{a}_{1, 1}|}{|\gcd(\mathfrak{a}_{1,1}, \mathfrak{m})|}
\left| \rho^{-1}(\Gamma_{\Scal}) \cap \Scal_{G(\gcd(\mathfrak{a}_{1, 1}, \mathfrak{m}))} \right|
\asymp_{M, K}
|\mathfrak{a}_{1,1}|.
\end{equation}

Putting together 
(\ref{Ja-degree-first}),
(\ref{Ga1/Scala1}),
(\ref{bound-a12-part}),
(\ref{Ga11-order}),
and
(\ref{Scala11-order}),
we obtain that
\begin{equation}\label{Ja-degree-second}
[J_{\mathfrak{a}} :K] 
\asymp_K
|\mathfrak{a}_{1, 1}|^{r^2 - 1}  \left|\PGL_r(A/\mathfrak{a}_2)\right|.
\end{equation}

To conclude the proof, observe that
\begin{eqnarray*}\label{PGL-a2}
 \left|\PGL_r(A/\mathfrak{a}_2)\right|
 &=&
|\mathfrak{a}_2|^{r^2 - 1}
\ds\prod_{
\mathfrak{p} \mid \mathfrak{a}_2
\atop{\mathfrak{p} \ \text{prime}}
}
\left(1 - \frac{1}{\left|\mathfrak{p}\right|^2}\right)
\ldots
\left(1 - \frac{1}{\left|\mathfrak{p}\right|^r}\right)
\\
&\geq&
|\mathfrak{a}_2|^{r^2 - 1}
\ds\prod_{
\mathfrak{p}
\atop{\mathfrak{p} \ \text{prime}}
}
\left(1 - \frac{1}{\left|\mathfrak{p}\right|^2}\right)
\ldots
\left(1 - \frac{1}{\left|\mathfrak{p}\right|^r}\right)
\\
&\gg_{r, K}&
|\mathfrak{a}_2|^{r^2 - 1},
\end{eqnarray*}
which, combined with (\ref{Ja-degree-second}), (\ref{factor-a}) and (\ref{def-a12}),
gives
\begin{eqnarray*}
[J_{\mathfrak{a}} :K] 
\asymp_{K}
\frac{|\mathfrak{a}_{1}|^{r^2 - 1}}{|\mathfrak{a}_{1, 2}|^{r^2 - 1}}  
\left|\PGL_r(A/\mathfrak{a}_2)\right|
\gg_{r, K}
\frac{|\mathfrak{a}_{1} \mathfrak{a}_2|^{r^2 - 1}}{|\mathfrak{a}_{1, 2}|^{r^2 - 1}}  
\gg_{M, K}
|\mathfrak{a}|^{r^2 -1}.
\end{eqnarray*}

\section{Proof of Corollaries \ref{Cor2} and \ref{Cor3}}

First consider the setting of Corollary \ref{Cor2}: 
 $K$  a finitely generated field with $\car K = 0$
and 
$E/K$ an elliptic curve over $K$ with $\End_{\overline{K}}(E) \simeq \Z$.
This is the specialization to the setting of Theorem \ref{Thm1} to
$k = \Q$, $A = \Z$, $K$ as above, and $M = E$. 
In this case, $r = 2$
and assumption (\ref{assumption}) holds thanks to an extension of Serre's Open Image Theorem for elliptic curves
over number fields \cite[Th\'{e}or\`{e}me 3, p. 299]{Se72} as explained in \cite[Theorem 3.2, p. 1248]{Br10}.
Corollary \ref{Cor2} follows.

Next consider the setting of Corollary \ref{Cor3}:
$k$ a global function field, 
$\infty$ a fixed place of $k$,
$A$ the ring of elements of $k$ regular away from $\infty$,
$K$ a finitely generated $A$-field with $\car K = \car k$ and $A$-$\car K = 0$,
and 
$\psi: A \longrightarrow K\{\tau\}$ a Drinfeld $A$-module over $K$ of rank $r \geq 2$
with $\End_{\overline{K}}(\psi) \simeq A$.
This is the specialization to the setting of Theorem \ref{Thm1} to
$k$, $A$, $K$ as above, and $M = \psi$. 
In this case, assumption (\ref{assumption}) holds thanks to Pink-R\"{u}tshe's Open Image Theorem for Drinfeld modules \cite[Theorem 0.1, p. 883]{PiRu09}.
Corollary \ref{Cor3} follows.


{\small{

}}

\end{document}